\documentclass[12pt]{amsart}
\usepackage{amssymb,amsmath,amsthm,amssymb}
\usepackage{mathtools}
\usepackage{graphicx}
\usepackage{comment}
\usepackage{hyperref}
\usepackage{enumerate}
\usepackage{color}
\usepackage{epstopdf}
\renewcommand{\epsilon}{\varepsilon}
\DeclareMathOperator{\spt}{spt}
\DeclareMathOperator{\dist}{dist}

\DeclareMathOperator{\graph}{graph}

\DeclareMathOperator{\Span}{span}
\DeclareMathOperator{\Div}{div}
\renewcommand{\div}{\Div}

\DeclareMathOperator{\sech}{sech}

\def\R{\mathbb{R}}

\def\N{\mathbb{N}}

\def\a{\alpha}
\def\e{\epsilon}

\def\r{\rho}

\def\s{\sigma}
\def\o{\omega}
\def\l{\lambda}

\def\H{\mathcal{H}}

\def\wt{\widetilde}

\def\ov{\overline}

\newcommand{\pd}{\partial}
\newcommand{\cd}{\nabla}

\def\ba #1\ea {\begin{align} #1\end{align}}
\def\bann #1\eann {\begin{align*} #1\end{align*}}
\def\ben #1\een {\begin{enumerate} #1\end{enumerate}}
\def\bi #1\ei {\begin{itemize}\renewcommand\labelitemi{--} #1\end{itemize}}

\newcommand{\inner}[2]{\left\langle#1,#2\right\rangle} 

\newtheorem{mtheorem}{Theorem}

\newtheorem{theorem}{Theorem}[section]
\newtheorem*{theorem*}{Theorem}
\newtheorem{lemma}[theorem]{Lemma}

\newtheorem{thm}[theorem]{Theorem}
\newtheorem{remark}[theorem]{Remark}
\newtheorem*{remark*}{Remark}
\newtheorem{corollary}[theorem]{Corollary}
\newtheorem*{corollary*}{Corollary}

\newtheorem*{conjecture*}{Conjecture}

\newtheorem{claim}{Claim}[theorem]

\newtheoremstyle{TheoremNum}
        {\topsep}{\topsep}              
        {\itshape}                      
        {}                              
        {\bfseries}                     
        {.}                             
        { }                             
        {\thmname{#1}\thmnote{ \bfseries #3}}
    \theoremstyle{TheoremNum}

\title[Translating solutions of MCF in slab regions]{On the existence of translating solutions of mean curvature flow in slab regions}
\author{Theodora Bourni}
\author{Mat Langford}
\author{Giuseppe Tinaglia}
\thanks{The third author was partially supported by EPSRC grant no. EP/M024512/1}

\address{Department of Mathematics, University of Tennessee Knoxville, Knoxville TN, 37996-1320}
\email{tbourni@utk.edu}
\email{mlangford@utk.edu}
\address{Department of Mathematics, King's College London, London, WC2R 2LS, U.K.}
\email{giuseppe.tinaglia@kcl.ac.uk}

\begin{document}

\begin{abstract}
We prove, in all dimensions $n\geq 2$, that there exists a convex translator lying in a slab of width $\pi\sec\theta$ in $\R^{n+1}$ (and in no smaller slab) if and only if $\theta\in[0,\frac{\pi}{2}]$. We also obtain convexity and regularity results for translators which admit appropriate symmetries and study the asymptotics and reflection symmetry of translators lying in slab regions.
\end{abstract}

\maketitle

\tableofcontents

\section{Introduction}

A solution of mean curvature flow is a smooth one-parameter family $\{\Sigma_t\}_{t\in\R}$ of hypersurfaces $\Sigma_t$ in $\R^{n+1}$ with normal velocity equal to the mean curvature vector. A translating solution of mean curvature flow is one which evolves purely by translation: $\Sigma_{t+s}=\Sigma_t+se$ for some $e\in \R^{n+1}\setminus \{0\}$ and each $s,t\in(-\infty,\infty)$. In that case, the time slices are all congruent and 
satisfy
\begin{equation}\label{eq:translator}
H=-\left\langle \nu,e\right\rangle\,,
\end{equation}
where $\nu$ is a choice of unit normal field and $H=\Div\nu$ is the corresponding mean curvature. Conversely, if a hypersurface satisfies \eqref{eq:translator} then the one-parameter family of translated hypersurfaces $\Sigma_t:=\Sigma+te$ satisfies mean curvature flow. We shall eliminate the scaling invariance and isotropy of \eqref{eq:translator} by restricting attention to translating solutions which move with unit speed in the `upwards' direction. That is, we henceforth assume that $e=e_{n+1}$. We will refer to a hypersurface $\Sigma^n\subset \R^{n+1}$ satisfying \eqref{eq:translator} with $e=e_{n+1}$ as a \emph{translator}.

X.-J.~Wang proved that any convex translator in $\R^{n+1}$ which is not an entire graph must lie in a slab region \cite[Corollary 2.2]{Wa11}. 
Using the Grim hyperplane, described below, as a barrier, it can be shown that there can exist no strictly convex translator in a slab of width less than or equal to $\pi$. The main result of this paper provides the existence of a strictly convex translator in all larger slabs.

Denote by $S_\theta^{n+1}\subset \R^{n+1}$ be the slab region defined by
\[
S_\theta^{n+1}:=\{(x,y,z)\in \R\times\R^{n-1}\times \R:|x|<\tfrac{\pi}{2}\sec\theta\}\subset \R^{n+1}\,.
\]

\begin{mtheorem}[Existence of convex translators in slab regions]\label{thm:existence}
For every $n\geq 2$ and every $\theta\in(0,\frac{\pi}{2})$ there exists a strictly convex translator $\Sigma^n_\theta$ which lies in $S_\theta^{n+1}$ 
and in no smaller slab. 
\end{mtheorem}

Spruck and Xiao have recently proved that every mean convex translator in $\R^3$ is actually convex \cite[Theorem 1.1]{SX}. To prove convexity of the examples in Theorem \ref{thm:existence}, we have extended their result to higher dimensions under a rotational symmetry hypothesis. 

Around the same time our work was completed, Hoffman, Ilmanen, Martin and White have also provided an existence theorem for all slabs of width greater than $\pi$ in the case $n=2$ \cite[Theorem 1.1]{HIMW}. They are also able to prove uniqueness in this case, thereby completing the classification of translating graphs in $\mathbb R^3$.


The most prominent example of a translator is the Grim Reaper curve, $\Gamma^1\subset \R^2$, defined by
\[
\Gamma^1:=\left\{(x,-\log\cos x):\vert x\vert<\tfrac{\pi}{2}\right\}\,.
\]
Taking products with lines then yields the Grim hyperplanes
\[
\Gamma^n:=\left\{(x_1,\dots,x_n,-\log\cos x_1):\vert x_1\vert<\tfrac{\pi}{2}\right\}\,.
\]
The Grim hyperplane $\Gamma^n$ lies in the slab $\{(x_1,\dots,x_n):\vert x_1\vert<\frac{\pi}{2}\}$ (and in no smaller slab). More generally, if $\Sigma^{n-k}$ is a translator in $\R^{n-k+1}$ then $\Sigma^{n-k}\times \R^k$ is a translator in $\R^{n-k+1}\times \R^k\cong \R^{n+1}$. 

There is also a family of `oblique' Grim planes $\Gamma_{\theta,\phi}^n$ parametrized by $(\theta,\phi)\in [0,\frac{\pi}{2})\times S^{n-2}$. These are obtained by rotating the `standard' Grim plane $\Gamma^n$ through the angle $\theta\in[0,\frac{\pi}{2})$ in the plane $\Span\{\phi,e_{n+1}\}$ for some unit vector $\phi\in\Span\{e_2,\dots e_{n}\}$ and then scaling by the factor $\sec\theta$. To see that the result is indeed a translator, we need only check that
\[
-H_\theta=-\cos\theta H=\cos\theta\inner{\nu}{e_{n+1}}=\inner{\cos\theta\nu+\sin\theta\phi}{e_{n+1}}=\inner{\nu_\theta}{e_{n+1}}\,,
\]
where $H_\theta$ and $\nu_\theta$ are the mean curvature and  unit normal to $\Gamma^n_{\theta,\phi}$ respectively.
The oblique Grim hyperplane $\Gamma_{\theta,\phi}^n$ lies in the slab $S^{n+1}_\theta:=\{(x_1,\dots,x_n):\vert x_1\vert<\frac{\pi}{2}\sec\theta\}$ (and in no smaller slab). More generally, if $\Sigma^{n-k}$ is a translator in $\R^{n-k+1}$ then the hypersurface $\Sigma^n_{\theta,\phi}$ obtained by rotating $\Sigma^{n-k}\times \R^k$ counterclockwise through angle $\theta$ in the plane $\phi\wedge e_{n+1}$ and then scaling  by $\sec\theta$ is a translator in $\R^{n+1}$, so long as $\phi$ is a non-zero vector in $\Span\{e_{n-k+1},\dots e_{n}\}$. 

Altschuler and Wu constructed a convex translating entire graph asymptotic to a paraboloid \cite{AlWu94} (see also \cite{CSS07}). X.-J.~Wang proved that this solution is the only convex entire translator in $\R^3$ and constructed further convex entire examples in higher dimensions \cite{Wa11}. He also proved the existence of strictly convex translating solutions which lie in slab regions in $\R^{n+1}$ for all $n\geq 2$. The existence of these examples was obtained by exploiting the Legendre transform and the existence of convex solutions of certain fully nonlinear equations. Unfortunately, this method loses track of the precise geometry of the domain on which the solution is defined and so it remained unclear exactly which slabs admit translators (cf. \cite[Remark 1.6]{SX}). Theorem~\ref{thm:existence} resolves this problem.

Hoffman, Ilmanen, Martin and White provide a different construction of examples of translating graphs in slabs in $\mathbb R^{n+1}$, extending an earlier (unpublished) construction of Ilmanen for the case $n=2$ \cite{HIMW}. However, just as for X.-J. Wang's examples, it is unclear exactly in which slabs these translators lie.


Our next theorem provides a step towards addressing the classification problem in higher dimensions.
 
 \begin{mtheorem}[Unique asymptotics and reflection symmetry]\label{thm:asymptotics}
Given $n\geq 2$ and $\theta\in(0,\frac{\pi}{2})$ let $\Sigma_\theta^n$ be a convex translator which lies in $S_\theta^{n+1}$ 
and in no smaller slab. If $n\geq 3$, assume in addition that $\Sigma_\theta^n$ is rotationally symmetric with respect to the subspace $\mathbb{E}^{n-1}:=\Span\{e_2,\dots,e_n\}$. Given any unit vector $\phi\in\mathbb{E}^{n-1}$ the curve $\{\sin\omega \phi-\cos\omega e_{n+1}:\omega\in[0,\theta)\}$ lies in the normal image of $\Sigma_\theta^n$ and the translators
\[
\Sigma^n_{\theta,\omega}:=\Sigma^n_\theta-P(\sin\omega \phi-\cos\omega e_{n+1})
\]
converge locally uniformly in the smooth topology to the oblique Grim hyperplane $\Gamma^n_{\theta,\phi}$ as $\omega\to \theta$, where $P:S^n\to\Sigma^n_{\theta}$ is the inverse of the Gauss map.

Moreover, $\Sigma_\theta^n$ is reflection symmetric across the hyperplane $\{0\}\times \R^{n}$.
\end{mtheorem}

We note that this result was already obtained by Spruck and Xiao when $n=2$ using different methods \cite{SX}. Furthermore, the translators we construct in this paper satisfy the hypotheses of Theorem~\ref{thm:asymptotics} and thus Theorem~\ref{thm:asymptotics} can be applied to show that they are reflection symmetric across the midplane of the slab and asymptotic to the `correct' oblique Grim hyperplanes.

The rotational symmetry hypothesis --- which is not required when $n=2$ --- may be necessary in higher dimensions: It is conceivable that there exist convex translators in the slab $S^4_\theta\subset\R^4$, for example, which are asymptotic to an `oblique' $\Sigma_\theta^2\times \R$, where $\Sigma^2_\theta\subset \R^3$ is the translator from Theorem \ref{thm:existence}.

\section*{Acknowledgements}
We would like to thank Joel Spruck and Ling Xiao for helpful discussions about their work.

\section{Compactness}

Recall that, given a smooth function $u$ over a domain $\Omega\subset\R^n$, the downward pointing unit normal $\nu$ and the mean curvature $H[u]$ of $\graph u$ are given by
\[
\nu=\frac{(Du,-1)}{\sqrt{1+|Du|^2}}\quad\text{and}\quad H[u]=\div\left(\frac{Du}{\sqrt{1+|Du|^2}}\right)
\]
respectively. So $\graph u$ is a translator (possibly with boundary) when
\begin{equation}\label{eq:graph_translator}
\div\left(\frac{Du}{\sqrt{1+|Du|^2}}\right)=\frac{1}{\sqrt{1+|Du|^2}}\,.
\end{equation}

In this section, we will derive uniform $C^{1,\alpha}$  estimates for certain hypersurfaces that are given by the the graphs of the Dirichlet problem,
\begin{equation}\label{eq:DP}
\begin{split}
\Div\left(\frac{Du}{\sqrt{1+|Du|^2}}\right)&=\frac{1}{\sqrt{1+|Du|^2}}\;\text{ in }\;\Omega\\
u&=\psi\;\text{ on }\;\partial \Omega\,,
\end{split}
\end{equation}

where $\Omega$ is a bounded open convex subset of $\R^{n+1}$ with $C^{1,\a}$ boundary and $\psi:\partial\Omega\to \R$ is a $C^{1,\a}$ function for some $\a\in (0,1]$. 



By Allard's regularity Theorems \cite{Allardinterior,Allardboundary, Bourni1}, the desired estimates are a consequence of the following lemma, where the usual dimension restriction is circumvented here due to the rotational symmetry of the solutions (cf. Remark \ref{rmk:Simons} below).

\begin{lemma}\label{lem:regularity}
Given any $\e, K>0$ there exists $\l_0=\l_0(\e,K)$ with the following property: Let $u$ be a solution of \eqref{eq:DP}, with $\partial\Omega$ and $\psi$ being rotationally symmetric with respect to the subspace $\mathbb{E}^{n-1}:=\Span\{e_2,\dots,e_n\}$ and having $C^{1,\a}$ norms bounded by $K$. For any $p\in\graph u$ and $\l\le\l_0$ the following statements hold.

If $  B^{n+1}_{\l }(p)\cap \graph \psi=\emptyset$ then
\begin{equation}\label{P-close-area}
\o_n^{-1}\l^{-n}\H^n(\graph u\cap B_\l^{n+1}(p))\le 1+\e\,.
\end{equation}

If $p\in\graph \psi$ then 
\begin{equation}\label{P-close bdry}
\o_n^{-1}\l^{-n}\H^n(\graph u\cap B_\l^{n+1}(p))\le \frac12+\e\,.
\end{equation}
\end{lemma}

\begin{proof}

Arguing by contradiction, assume that the conclusion is not true. Then there exist $\e_0>0$ and $K_0>0$,  sequences of rotationally symmetric domains $\Omega_i$ and boundary data $\psi_i:\pd\Omega_i\to\R$  bounded in $C^{1,\alpha}$ by $K_0$, corresponding solutions $u_i$ of the Dirichlet problem \eqref{eq:DP}, points $p_i\in\graph u_{i}$ and scales $\l_i\downarrow 0$ such that   \eqref{P-close-area}  (or  \eqref{P-close bdry} in case $p_i\in \graph \psi_i$) of the lemma with this $\e_0$ and with $u=u_i$, $p=p_i$ and $\l=\l_i$ fails for all $i$, namely
\begin{equation}\label{contr}
\o_n^{-1}{\l_i}^{-n}\H^n(\graph u_i\cap B_{\l_i}^{n+1}(p_i))> 1+\e_0\,.
\end{equation}

Set $\wt{\Omega}_i=\eta_{p_i,\l_i}(\Omega_i)$, $\Psi_i:=\graph\psi_i$ and $\wt{\Psi}_i=\eta_{p_i,\l_i}(\Psi_i)$, where $\eta_{p,\l}(y)=\l^{-1}(y-p)$. We define the current  $\wt{T}_i=\eta_{p_i,\l_i\#}(T_i)$, where $T_i=[\![\graph u_i]\!]$ and note that $\wt{T}_i=[\![\graph\wt{u}_i]\!]$, where $\wt{u}_i\in C^{1,\a}(\wt{\Omega}_i)$ is defined by $\wt{u}_i(p)=\eta_{p_i,\l_i}(u_i(\l_ip+p_i))$ and whose mean curvature satisfies
\[\wt{H}_i(p)=\l_i H_i(x_i+\l_ip)\le \l_i\implies \|\wt{H}_i\|_{0, \wt \Omega_i}\stackrel{i\to\infty}{\longrightarrow}  0.\]
It follows, after passing to a subsequence, that
\begin{enumerate}[(i)]
\item \label{item:currentconvergence} $\wt{T}_i\to T$ in the weak sense of currents, where $T$ is area minimizing;
\item \label{item:measureconvergence} $\mu_{\wt{T}_i}\to\mu_T$ as Radon measures, where $\mu_{\wt T_i}$ and $\mu_T$ denote the total variation measures of $\wt T_i$ and $T$ respectively.
\end{enumerate}
See for instance~\cite[Lemma 2.15]{Bourni} or~\cite[Theorem 34.5]{Simon} for details.

By the measure convergence (\ref{item:measureconvergence}), for every $\e>0$ there exists $i_0$ such that, for all $i\ge i_0$,
\[
\l_i^{-n}\mu_{T_i}(B^{n+1}_{\l_i}(p_i))=\mu_{\wt{T}_i}(B^{n+1}_1(0))\le \mu_{T} (B^{n+1}_1(0))+\e\,.
\]

We consider the following three cases for the sequence of points $p_i$.
\begin{enumerate}
\item $p_i\in \Psi_i=\partial \graph u_i$. 
\item $p_i=(x_i, y_i, u(x_i, y_i))\notin \Psi_i$, $y_i\in \R^{n-1}$ with $|y_i|=0$ for all $i$ and $\liminf_i\dist(p_i, \Psi_i)\neq0$.
\item $p_i=(x_i, y_i, u(x_i, y_i))\notin \Psi_i$, $y_i\in \R^{n-1}$ with $\liminf_i|y_i|\neq0$ and $\liminf_i\dist(p_i, \Psi_i)\neq0$.
\end{enumerate}

To prove Lemma~\ref{lem:regularity} it suffices to show the following claim.

\begin{claim}~\label{plane}
The current $T$ is  a hyperplane, when Case (2) or (3) holds, or a half-hyperplane, when Case (1) holds. 
\end{claim}
\begin{proof}
We first prove that $T$ is a hyperplane when Case (2) holds, that is when $p_i=(x_i, y_i, u(x_i, y_i))\notin \Psi_i$ with   $y_i=0\in \R^{n-1}$ and $\liminf_i\dist(p_i, \Psi_i)\neq0$. In this case the support of the area minimizing current $T$ is rotationally symmetric in the $y$-space. Note that as a consequence of the divergence theorem applied to the normals of the graphs (extended to be independent of the $e_{n+1}$ direction) in an appropriately chosen domains, the following lemma is true (see \cite[Lemmas 2.10, 2.12]{Bourni} for a proof).

\begin{lemma}\label{area-volume}
There exists a constant $c$ such that for any $i$,  $p\in \ov \Omega_i\times\R$ and $\r>0$ the following hold:
Let  $\s\in(0,1)$, $Q_{\r,\s}=[-\s\r,\s\r]\times B_\r^n(0)$ and $q$ an orthogonal transformation of $\R^{n+1}$ such that $q(0)=p$. Then
\[
\H^n(\graph u_i\cap q(Q_{\r,\s}))\le \o_n\r^{n}(1+c\s(n+\r\|H_i\|_{0})).\]
\end{lemma}

By the nature of the measure convergence, Lemma~\ref{area-volume} and the interior monotonicity formula~\cite{Allardinterior} (see also \cite[Section 17]{Simon}) we obtain that

\begin{equation}\label{arearatios}
\begin{split}
1
\le \o_n^{-1}r^{-n}\mu_T(B^{n+1}_r(p))&=\o_n^{-1}r^{-n}\lim_i\mu_{{\wt T}_i}(B^{n+1}_r(p))\\
&=\o_n^{-1}(\l_ir)^{-n}\lim_i\mu_{{ T}_i}(B^{n+1}_{\l_ir}(p))\le 1+cn
\end{split}
\end{equation}
for all $p\in\spt T$ and any $r>0$, where $c$ is  independent of $i$.
Thus, for a sequence $\{\Lambda_k\}\uparrow\infty$, we can apply the Federer--Fleming compactness theorem \cite{FeFl60} (see also \cite[Theorem 32.2]{Simon}) to the sequence $T_{0,\Lambda_k}=\eta_{0,\Lambda_k\#}T$; after passing to a subsequence, this yields $T_{0,\Lambda_k}\to C$ 
in the weak sense of currents, where $\spt C$ is an area minimizing cone rotationally symmetric in the $y$-space, and $\mu_{T_{p,\Lambda_k}}\to \mu_C$ as radon measures.  Since $\spt C\cap S^n$ is an embedded minimal surface in $S^n$ which has at most two distinct principal curvatures at each point, by Theorem 1.5 in~\cite{AnHuLi15} it must be either $S^{n-1}$ or the Clifford torus, $S^1_{\frac{1}{\sqrt{n-1}}}\times S^{n-2}_{\sqrt{\frac{n-2}{n-1}}}$. The latter is not possible because it cannot be obtained from a limit of graphs, so we may conclude that, after applying a rigid motion, $C=m[\![\R^n\times\{0\}]\!]$, with $m\in\mathbb N$. We claim that in fact $m=1$.

Arguing by contradiction, assume that $m\geq2$. Then $\mu_C(Q_{1,\s})=m\o_n\geq 2\o_n$, where  $Q_{1,\s}$ is defined in Lemma~\ref{area-volume}. Let $\sigma=\frac1{2cn}$ in Lemma~\ref{area-volume} then, by the nature of the measure convergence, we have that 
\[
2\o_n\leq \mu_C(Q_{1,\s})\leq \o_n(1 +\frac12).
\]

This contradiction proves that $m=1$ and $C= \R^n\times\{0\}$. Finally, since $C= \R^n\times\{0\}$, the monotonicity formula applied to $T$ gives that
\[
\o_n^{-1}r^{-n}\mu_T(B^{n+1}_r(0))=1\text {  for all  }r>0
\]
which implies that $T$ itself is a hyperplane. This finishes the proof of Claim~\ref{plane} when Case (2) holds. 

We now prove that $T$ is a half-hyperplane when Case (1) holds, that is when $p_i\in \Psi_i=\partial \graph u_i$. The idea of the proof is similar to the proof of Case (2) and the details can be found in~\cite[Lemma 2.15, Lemma 2.16]{Bourni}. For completeness, we provide a sketch. 

Since $\Omega_i$ is a convex domain and the $C^{1,\a}$ norms of $\partial\Omega_i$ and $\psi_i$ are uniformly bounded by $K_0$, it follows that $T$ is an area minizing current contained in a half-hyperspace with boundary a straight line. Applying the Federer--Fleming compactness theorem we obtain a cone at infinity $C$ that satisfies the same property that $T$ satisfies, namely $C$ is contained in a half-hyperspace with boundary a straight line. This being the case, a result of Allard \cite{Allardinterior} (see also \cite[Appendix B]{Bourni}) shows that $C$ is a half-hyperplane with multiplicity $m$ and, as in the proof of Case (2), it remains to show that $m=1$. This follows using the same arguments that we have used in Case (2) after noting that the estimates described in Lemma~\ref{area-volume} can be improved when $p$ is a boundary point. Namely, 
\[
\H^n(\graph u_i\cap q(Q_{\r,\s}))\le \o_n\r^{n}\left(\frac12+c\s(n+\r\|H_i\|_{0})\right).\]
This finishes the sketch of the proof of Claim~\ref{plane} when Case (1) holds.

It remains to prove that $T$ is a hyperplane when Case (3) holds, that is $p_i=(x_i, y_i, u(x_i, y_i))\notin \Psi_i$, $y_i\in \R^{n-1}$ with $\liminf_i|y_i|\neq0$ and $\liminf_i\dist(p_i, \Psi_i)\neq0$.   After passing to a subsequence we can assume that $\lim|y_i|=|y_\infty|$ exists, with $|y_\infty|\in (0, \infty]$. Rotational symmetry of $\graph u_i$ in the $y$-space then implies that $T=[\![\R^{n-2}]\!]\times T_0$, 
where $T_0$ is an area minimizing 2-current in $\R^3$. Since any such current is regular,  $T_0$, and hence also $T$, is regular. We conclude that $\spt T_0$ must be a plane \cite{cp1,po1,sc3} with (arguing as in Case (2)) multiplicity one. This finishes the proof of Claim~\ref{plane} when Case (3) holds and indeed, the proof of Claim~\ref{plane} is finished.
\end{proof}

Recall by the measure convergence, for every $\e>0$ there exists $i_0$ such that, for all $i\ge i_0$,
\[
\l_i^{-n}\mu_{T_i}(B^{n+1}_{\l_i}(p_i))=\mu_{\wt{T}_i}(B^{n+1}_1(0))\le \mu_{T} (B^{n+1}_1(0))+\e\,.
\]
If either Case (2) or (3) holds, then $T$ is a hyperplane and 
\[
\l_i^{-n}\mu_{T_i}(B^{n+1}_{\l_i}(p_i))=\mu_{\wt{T}_i}(B^{n+1}_1(0))\le \o_n+\e\,.
\]
By taking $\e<\e_0$ this contradicts \eqref{contr}. A similar contradiction can be obtained when Case (1) holds. This finishes the proof of Lemma~\ref{lem:regularity}.
 \end{proof}
 
Lemma \ref{lem:regularity} allows us to apply Allard's interior and boundary regularity theorems \cite{Allardinterior,Allardboundary, Bourni1} to obtain uniform $C^{1,\a}$ estimates for the graphs of solutions $u$ to \eqref{eq:DP} with boundary data that satisfy the hypotheses of Lemma \ref{lem:regularity}. Assuming higher regularity of the boundary data we can apply the  Schauder theory to obtain higher regularity estimates for these graphs.
\begin{corollary} \label{cor:regularity}
Given any $K>0$ and $\ell\in \N$,  there exists a constant $C_\ell$ with the following property: Let $u$ be a solution of \eqref{eq:DP} with $\partial\Omega$ and $\psi$ bounded in $C^{\ell_0,\a}$ by $K$ for some $\ell_0\ge 2$ and $\a\in (0,1]$ and rotationally symmetric with respect to the subspace $\mathbb{E}^{n-1}:=\Span\{e_2,\dots,e_n\}$. Then
\[
\sup_{p\in \graph u}|\nabla ^\ell A(p)|\le C_\ell\quad\text{for all}\quad \ell \in \{0,\dots,\ell_0-2\}\,,
\]
where $A$ is the second fundamental form of $\graph u$ and $\nabla^0A:=A$.
\end{corollary}
\begin{remark} \label{rem:regularity}
If we allow $\ell_0=1$ in Corollary \ref{cor:regularity} then we obtain uniform $C^{1,\a}$ estimates for the graphs of solutions $u$ to \eqref{eq:DP} with boundary data that satisfy the hypotheses of Lemma \ref{lem:regularity}.
\end{remark}

\begin{remark}\label{rmk:Simons} If $n+1\le 7$ then Lemma \ref{lem:regularity}, and hence Corollary \ref{cor:regularity} and Remark \ref{rem:regularity}, still hold without the rotational symmetry hypothesis on the boundary data. To see this, note that the proof of the boundary case (Case 1) of Lemma \ref{lem:regularity} does not make use of the rotational symmetry hypothesis and hence holds in all dimensions without this restriction. To show interior regularity in case $n+1\le 7$ we can refer to known results on regularity of almost minimizing surfaces; see for example \cite{DS93, MaMi}. One can alternatively see this from Cases 2a and 2b in the proof of Lemma \ref{lem:regularity}, since there are no stable non-planar minimal cones in low dimensions \cite{Si68}  (see also \cite{SSY} or \cite[Appendix B]{Simon}).
\end{remark}

\section{Convexity}\label{sec:convexity}

In this section, we prove a generalization of  extend the convexity result of Spruck and Xiao \cite[Theorem 1.1]{SX}. Our proof is a straightforward extension of theirs.

\begin{thm}\label{thm:convexity}
Let $\Sigma\subset \R^{n+1}$ be a mean convex translator with at most two distinct principal curvatures at each point and bounded second fundamental form. Then $\Sigma$ is convex.
\end{thm}
\begin{proof}
Denote the principal curvatures of $\Sigma$ by $\kappa\leq \mu$. Note that $\kappa$ is smooth and has constant multiplicity $m\in\{1,\dots,n-1\}$ in the open set $U:=\{X\in \Sigma:\kappa(X)<0\}$. Recall that
\[
-(\cd_V+\Delta)A=|A|^2A\,,
\]
where $V:=e_{n+1}^\top$ is the tangential part of $e_{n+1}$. Computing locally in a principal frame $\{\tau_1,\dots,\tau_n\}$ with $\kappa_i=A_{ii}=\kappa$ when $i\leq m$ and $\kappa_i=A_{ii}=\mu$ when $i\geq m+1$, we obtain
\[
-(\cd_V+\Delta)\kappa=|A|^2\kappa+2\sum_{\ell=1}^n\sum_{p=m+1}^n\frac{(\cd_\ell A_{1p})^2}{\mu-\kappa}\quad\text{in}\quad U\,.
\]
Since the mean curvature satisfies
\[
-(\cd_V+\Delta)H=|A|^2H\,,
\]
straightforward manipulations then yield
\ba\label{eq:SMP}
-(\cd_V+\Delta)\frac{\kappa}{\mu}={}&-(\cd_V+\Delta)\frac{(n-m)\kappa}{H-m\kappa}\nonumber\\
={}&\frac{2}{n-m}\frac{H}{\mu^2}\sum_{\ell=1}^n\sum_{p=m+1}^n\frac{(\cd_\ell A_{1p})^2}{\mu-\kappa}+2\inner{\cd\frac{\kappa}{\mu}}{\frac{\cd\mu}{\mu}}\,.
\ea
Suppose that
\[
-\varepsilon_0:=\inf_{\Sigma}\frac{\kappa}{\mu}<0\,.
\]
If the infimum is attained at some point $X_0\in \Sigma$ then $\kappa(X_0)<0$ and the strong maximum principle yields $\frac{\kappa}{\mu}\equiv -\varepsilon_0<0$. In particular,
\[
0\equiv \cd_\ell\frac{\kappa}{\mu}=\frac{\cd_\ell A_{pp}}{\kappa}-\frac{\kappa}{\mu}\frac{\cd_\ell A_{qq}}{\mu}
\]
when $p\leq m<q$. It is a general observation that
\ba\label{eq:twocurvatures}
0=\tau_\ell A_{ij}=\cd_\ell A_{ij}+(\kappa_j-\kappa_i)\Gamma_{\ell ij}=\cd_\ell A_{ij}
\ea
for each $\ell$ whenever $\kappa_i=\kappa_j$ and $i\neq j$, where $\Gamma_{\ell ij}:=\inner{\cd_\ell\tau_i}{\tau_j}$. Thus\footnote{Here, and elsewhere, we freely make use of the Codazzi identity.},
\[
0=\cd_\ell A_{11}\;\; \text{when}\;\; \ell =2,\dots,m
\]
and
\[
0=\cd_\ell A_{nn}\;\; \text{when}\;\; \ell=m+1,\dots,n-1\,.
\]
Recalling \eqref{eq:SMP}, we also find that
\[
0\equiv \sum_{\ell=1}^n\sum_{p=m+1}^n(\cd_\ell A_{1p})^2\,.
\]
It follows that the components $\cd_1A_{nn}$, $\cd_1A_{11}$, $\cd_nA_{11}$ and $\cd_nA_{nn}$ are all identically zero and hence, by the translator equation \eqref{eq:translator},
\[
0\equiv m\cd_\ell A_{11}+(n-m)\cd_\ell A_{nn}=\cd_\ell H=-\kappa_\ell\inner{\tau_\ell}{e_3}
\]
for each $\ell=1,\dots,n$. It follows that $\nu\equiv -e_3$ and hence $\Sigma$ is a hyperplane satisfying $H\equiv 1$, which is absurd.

Suppose then that the infimum is not attained. Since $\frac{\kappa}{\mu}\geq-\frac{n-m}{m}$ and the sectional curvatures of $\Sigma$ are bounded, the Omori--Yau maximum principle may be applied. This yields a sequence of points $X_i\to\infty$ such that
\ba\label{eq:OYMP}
\frac{\kappa}{\mu}(X_i)\to -\varepsilon_0\,,\quad \left\vert\cd\frac{\kappa}{\mu}(X_i)\right\vert\leq \frac{1}{i}\quad\text{and}\quad -\Delta\frac{\kappa}{\mu}(X_i)\leq \frac{1}{i}\,.
\ea
Consider the sequence of translators $\Sigma_i:=\Sigma-X_i$. Since $|A|$ is bounded, the translators $\Sigma_i$ converge locally uniformly in $C^\infty$, after passing to a subsequence, to a limit translator $\Sigma_\infty$. Note that, whenever $\kappa<0<\mu$,
\ba\label{eq:gradH}
\cd_\ell\frac{m\kappa}{\mu}={}&\frac{m\cd_\ell A_{11}}{\mu}-\frac{m\kappa}{\mu^2}\cd_\ell A_{nn}\nonumber\\
={}&\frac{\cd_\ell H}{\mu}-m\left(\frac{n-m}{m}+\frac{\kappa}{\mu}\right)\frac{\cd_\ell A_{nn}}{\mu}\,.
\ea
We claim that
\ba\label{eq:claim}
\left(\frac{n-m}{m}+\frac{\kappa}{\mu}(X_i)\right)\frac{\cd_kA_{nn}}{\mu}(X_i)\to 0\;\;\text{as}\;\;i\to\infty
\ea
for each $\ell=1,\dots,n$. Suppose that this is not the case. Then there exists $i_0\in\N$ and $\delta_0>0$ such that
\ba\label{eq:convexity_contra}
\left(\frac{n-m}{m}+\frac{\kappa}{\mu}(X_i)\right)\frac{\vert\cd_kA_{nn}\vert}{\mu}(X_i)>\delta_0
\ea
for all $i>i_0$ and some $k\in\{1,\dots,n\}$. By \eqref{eq:OYMP},
\[
\left(\frac{\cd_\ell A_{11}}{\mu}-\frac{\kappa}{\mu}\frac{\cd_\ell A_{nn}}{\mu}\right)(X_i)\to 0
\]
for each $\ell=1,\dots,n$ as $i\to\infty$ so that, replacing $\delta_0$ and $i_0$ if necessary,
\ba\label{eq:convexity_contra2}
\left(\frac{n-m}{m}+\frac{\kappa}{\mu}(X_i)\right)\frac{\vert\cd_kA_{11}\vert}{\mu}(X_i)>\delta_0
\ea
for all $i>i_0$. Moreover, by \eqref{eq:twocurvatures},
\[
\frac{\cd_\ell A_{nn}}{\mu}(X_i)\to 0
\]
as $i\to\infty$ for all $\ell=2,\dots,n-1$. So \eqref{eq:convexity_contra} (and hence also \eqref{eq:convexity_contra2}) holds with $k\in\{1,n\}$. Combining \eqref{eq:OYMP} and \eqref{eq:SMP} we obtain, at the point $X_i$,
\bann
\frac{1}{i}\geq{}&\frac{1}{n-m}\frac{H}{\mu-\kappa}\sum_{\ell=1}^n\sum_{p=m+1}^n\frac{(\cd_\ell A_{1p})^2}{\mu^2}+\inner{\cd\frac{\kappa}{\mu}}{\frac{\cd\mu}{\mu}}\\
={}&\frac{1}{n-m}\frac{H}{\mu-\kappa}\left(\sum_{p=m+1}^{n}\frac{(\cd_p\kappa)^2}{\mu^2}+(n-m)\frac{(\cd_1\mu)^2}{\mu^2}\right)\\
{}&+\cd_1\frac{\kappa}{\mu}\frac{\cd_1\mu}{\mu}+\sum_{\ell =2}^m\cd_\ell \frac{\kappa}{\mu}\frac{\cd_\ell \mu}{\mu}+\sum_{\ell =m+1}^n\cd_\ell \frac{\kappa}{\mu}\left(\frac{\cd_\ell \kappa}{\kappa}-\frac{\mu}{\kappa}\cd_\ell \frac{\kappa}{\mu}\right)\\
={}&-\frac{\mu}{\kappa}\sum_{\ell =m+1}^n\left(\cd_\ell \frac{\kappa}{\mu}\right)^2+\sum_{\ell =2}^m\cd_\ell \frac{\kappa}{\mu}\frac{\cd_\ell \mu}{\mu}+\cd_1\frac{\kappa}{\mu}\frac{\cd_1\mu}{\mu}+\frac{H}{\mu-\kappa}\frac{(\cd_1\mu)^2}{\mu^2}\\
{}&+\frac{\mu}{\kappa}\sum_{\ell =m+1}^n\cd_\ell \frac{\kappa}{\mu}\frac{\cd_\ell \kappa}{\mu}+\frac{1}{n-m}\frac{H}{\mu-\kappa}\sum_{\ell =m+1}^{n}\frac{(\cd_\ell \kappa)^2}{\mu^2}\\
\geq{}&-\frac{\mu}{\kappa}\sum_{\ell =m+1}^n\left(\cd_\ell \frac{\kappa}{\mu}\right)^2+\sum_{\ell =2}^m\cd_\ell \frac{\kappa}{\mu}\frac{\cd_\ell \mu}{\mu}\\
{}&+\left(m\frac{\frac{n-m}{m}+\frac{\kappa}{\mu}}{1-\frac{\kappa}{\mu}}\frac{\vert\cd_1\mu\vert}{\mu}-\left\vert\cd_1\frac{\kappa}{\mu}\right\vert\right)\frac{\vert\cd_1\mu\vert}{\mu}\\
{}&+\sum_{\ell =m+1}^{n}\left(\frac{m}{n-m}\frac{\frac{n-m}{m}+\frac{\kappa}{\mu}}{1-\frac{\kappa}{\mu}}\frac{\vert\cd_\ell \kappa\vert}{\mu}+\frac{\mu}{\kappa}\left\vert\cd_\ell \frac{\kappa}{\mu}\right\vert\right)\frac{\vert\cd_\ell \kappa\vert}{\mu}\,.
\eann
Suppose that $k=1$ in \eqref{eq:convexity_contra}. If
\[
\left(\frac{n-m}{m}+\frac{\kappa}{\mu}(X_i)\right)\frac{\vert\cd_n\kappa\vert}{\mu}(X_i)\not\to 0\;\;\text{as}\;\;i\to\infty
\]
then, taking $i\to\infty$, we find $\frac{\vert\cd_1\mu\vert}{\mu}(X_i)\to 0$ as $i\to\infty$, contradicting \eqref{eq:convexity_contra}. Else, $\frac{\vert\cd_n\kappa\vert}{\mu}(X_i)\leq \frac{\vert\cd_1\mu\vert}{\mu}(X_i)$ for $i$ sufficiently large and we again obtain $\frac{\vert\cd_1\mu\vert}{\mu}(X_i)\to 0$ as $i\to\infty$, contradicting \eqref{eq:convexity_contra}. If $k=n$ in \eqref{eq:convexity_contra} we may argue similarly, using \eqref{eq:convexity_contra2}.

So \eqref{eq:claim} does indeed hold. Applying \eqref{eq:OYMP} and \eqref{eq:claim} to \eqref{eq:gradH} yields
\[
\frac{\cd_\ell H}{\mu}(X_i)\to 0
\]
for each $\ell=1,\dots,n$. On the other hand, by the translator equation,
\[
\frac{\cd_\ell H}{\mu}=-\frac{\kappa_\ell\inner{\tau_\ell}{e_3}}{\mu}\,.
\]
Since $\frac{\kappa}{\mu}(X_i)\to-\varepsilon_0\neq 0$, we conclude that $\nu(X_i)\to -e_3$ and hence $H(X_i)\to 1$. So we may proceed as in the interior case to obtain a contradiction.
\end{proof}

\begin{remark}
A very similar argument can be used to prove that any two-convex translator (i.e. one satisfying $\kappa_1+\kappa_2>0$) in $\R^{n+1}$ with $n\geq 3$ is convex since in this case $|A|^2\leq nH^2\leq n$ and $\kappa_1$ is smooth and satisfies
\[
-(\cd_V+\Delta)\kappa_1=|A|^2\kappa_1+2\sum_{\ell=1}^n\sum_{\kappa_p>\kappa_1}\frac{(\cd_\ell A_{1p})^2}{\kappa_p-\kappa_1}
\]
wherever it is negative.
\end{remark}

\section{Barriers} In this section we introduce appropriate barriers. When $n=2$, the outer barrier is obtained by (non-isotropically) `stretching' the level set function corresponding to the Angenent oval so that it lies in the correct slab and is asymptotic to the correct oblique Grim planes. The higher dimensional barrier is then obtained by rotating in the $(n-1)$-dimensional complimentary subspace.

\begin{lemma}
The function $\underline u:\{(x,y)\in\R\times\R^{n-1}:|x|<\frac{\pi}{2}\sec\theta\}\to\R$ defined by
\[
\underline u(x,y):=-\sec^2\theta\log\cos\left(\tfrac{x}{\sec\theta}\right)+\tan^2\theta\log\cosh\left(\tfrac{|y|}{\tan\theta}\right)
\]
is a subsolution of the graphical translator equation \eqref{eq:graph_translator}.

In particular, given any $R>0$, the surface
\[
\underline \Sigma_{R}:=\graph \underline u_{R}
\]
is a subsolution of the translator equation \eqref{eq:translator}, where
\[
\underline u_{R}:=\underline u-\tan^2\theta\log\cosh\left(\tfrac{R}{\tan\theta}\right)\,.
\]
\end{lemma}
\begin{proof}
The relevant derivatives of $\underline u$ are given by
\[
D\underline u=\left(\sec\theta\tan\left(\tfrac{x}{\sec\theta}\right),\tan\theta\tanh\left(\tfrac{|y|}{\tan\theta}\right)\tfrac{y}{|y|}\right)
\]
and
\begin{equation*}
D^2\underline u=\left(\begin{matrix}
\sec^2\left(\tfrac{x}{\sec\theta}\right)&\dots \qquad\qquad 0 \qquad\qquad  \dots\\
\vdots &\\
0& \sech^2\left(\tfrac{|y|}{\tan\theta}\right)\frac{y_iy_j}{|y|^2}+\tan\theta\tanh\left(\tfrac{|y|}{\tan\theta}\right)\left(\frac{\delta_{ij}}{|y|}-\frac{y_iy_j}{|y|^3}\right)\\
\vdots &
\end{matrix}\right)\,.
\end{equation*}
So
\bann
1+|D\underline u|^2={}&1+\sec^2\theta\tan^2\left(\tfrac{x}{\sec\theta}\right)+\tan^2\theta\tanh^2\left(\tfrac{|y|}{\tan\theta}\right)\\
={}&\sec^2\theta\sec^2\left(\tfrac{x}{\sec\theta}\right)-\tan^2\theta\sech^2\left(\tfrac{|y|}{\tan\theta}\right)\,.
\eann
Estimating
\[
\Delta \underline u\geq \sec^2\left(\tfrac{x}{\sec\theta}\right)+\sech^2\left(\tfrac{|y|}{\tan\theta}\right)
\]
we find
\bann
(1+|D\underline u|^2)^{\frac{3}{2}}H[\underline u]={}&(1+|D\underline u|^2)\Delta \underline u-D^2\underline u(D\underline u,D\underline u)\\
\geq{}&1+|D\underline u|^2+\sec^2\left(\tfrac{x}{\sec\theta}\right)\sech^2\left(\tfrac{|y|}{\tan\theta}\right)\\
\geq{}&1+|D\underline u|^2\,.
\eann
\end{proof}

Consider the `outer' domain
\bann
\underline \Omega_{R}:={}&\left\{(x,y)\in S_\theta^{n-1}:\underline u_{R}(x,y)<0\right\}\nonumber\\
={}&\left\{(x,y)\in S_\theta^{n-1}:\cos\left(\tfrac{x}{\sec\theta}\right)<\left[\frac{\cosh\left(\tfrac{|y|}{\tan\theta}\right)}{\cosh\left(\tfrac{R}{\tan\theta}\right)}\right]^{\sin^2\theta}\right\}\,,
\eann
where $S_\theta^{n}:=(-\tfrac{\pi}{2}\sec\theta,\tfrac{\pi}{2}\sec\theta)\times\R^{n-1}$. Note that
\[
\pd \underline \Omega_{R}=\pd\left(\underline \Sigma_{R}\cap \R^n\times(-\infty,0]\right)\,.
\]

The inner barrier is obtained by rotating the Angenent oval of width $\pi\sec\theta$ and cutting off at an appropriate height (see Figure \ref{fig:supersolution}).

\begin{lemma}
Given $R>0$, let $\Pi_{R}\subset \R^{n+1}$ be the surface formed by rotating the time $T:=-\sec^2\theta\cosh\left(\frac{R}{\tan\theta}\right)$ slice of the Angenent oval of width $\pi\sec\theta$ about the $x$-axis. That is,
\[
\Pi_{R}:=\left\{(x,y,z)\in \R\times \R^{n-1}\times \R:v(x,y,z)=T\right\}\,,
\]
where
\[
v:=\sec^2\theta\left[\log\left(\cosh\left(\tfrac{\sqrt{|y|^2+z^2}}{\sec\theta}\right)\right)-\log\left(\cos\left(\tfrac{x}{\sec\theta}\right)\right)\right]\,.
\]
There exists $\varepsilon_0=\varepsilon_0(n,\theta)>0$ such that the sublevel set
\[
\overline \Sigma_{R,\varepsilon}:=\Pi_{R}\cap\left\{z\leq -R\tfrac{\cos(\theta-\varepsilon)}{\sin\theta}\right\}+R\tfrac{\cos(\theta-\varepsilon)}{\sin\theta} e_{n+1}
\]
is a supersolution of the translator equation \eqref{eq:translator} whenever $\varepsilon<\varepsilon_0$ and $R>R_\varepsilon:=\frac{2(n-1)}{\varepsilon}$.
\end{lemma}
\begin{proof}
Set $w=(y,z)$. Then
\[
Dv=\sec\theta \left(\tan\left(\tfrac{x}{\sec\theta}\right),\tanh\left(\tfrac{|w|}{\sec\theta}\right)\tfrac{w}{|w|}\right)
\]
and
\begin{equation*}
D^2v=\left(\begin{matrix}
\sec^2\left(\tfrac{x}{\sec\theta}\right)&\dots \qquad\qquad 0 \qquad\qquad  \dots\\
\vdots & \\
0& \sech^2\left(\tfrac{|w|}{\sec\theta}\right)\frac{w_iw_j}{|w|^2}+\sec\theta\tanh\left(\tfrac{|w|}{\sec\theta}\right)\left(\frac{\delta_{ij}}{|w|}-\frac{w_iw_j}{|w|^3}\right)\\
\vdots  &
\end{matrix}\right)\,.
\end{equation*}
Lengthy computations then yield, on the one hand,
\bann
-\inner{\nu}{e_{n+1}}=\inner{\frac{Dv}{|Dv|}}{e_{n+1}}={}&\frac{\tanh\left(\frac{\sqrt{|y|^2+z^2}}{\sec\theta}\right)\frac{|z|}{\sqrt{|y|^2+z^2}}}{\sqrt{\tan^2\left(\frac{x}{\sec\theta}\right)+\tanh^2\left(\frac{\sqrt{|y|^2+z^2}}{\sec\theta}\right)}}
\eann
and, on the other hand,
\bann
H={}&\div\left(\frac{Dv}{|Dv|}\right)
=\frac{\frac{1}{\sec\theta}+\frac{n-1}{|w|}\tanh\left(\tfrac{|w|}{\sec\theta}\right)}{\sqrt{\tan^2\left(\tfrac{x}{\sec\theta}\right)+\tanh^2\left(\tfrac{|w|}{\sec\theta}\right)}}\,.
\eann
It follows that $\Pi_{R}$ is a supersolution in the region where
\bann
\frac{|z|-(n-1)}{\sqrt{|y|^2+z^2}}\tanh\left(\tfrac{\sqrt{|y|^2+z^2}}{\sec\theta}\right)\geq \cos\theta\,.
\eann
Note that $|y|\leq \frac{\sin(\theta-\varepsilon)}{\sin\theta}R$ wherever $|z|\geq \frac{\cos(\theta-\varepsilon)}{\sin\theta}R$. Thus, whenever $R>R_\varepsilon:=\frac{2(n-1)}{\varepsilon}$ and $z\leq-\frac{\cos(\theta-\varepsilon)}{\sin\theta}R$,
\bann
\frac{|z|-(n-1)}{\sqrt{|y|^2+z^2}}{}&\tanh\left(\tfrac{\sqrt{|y|^2+z^2}}{\sec\theta}\right)\\
\geq{}&
\left(\cos(\theta-\varepsilon)-\tfrac{(n-1)}{R}\sin\theta\right)\tanh\left(\tfrac{\cos(\theta-\varepsilon)}{\tan\theta}R\right)\\
\geq
{}&\cos\theta\left(1+\tfrac{\varepsilon}{2}\tan\theta+o(\varepsilon)\right)\sqrt{1-4\mathrm{e}^{-\frac{2(n-1)\cos^2\theta\sin\theta}{\varepsilon}}}\,.
\eann
This is no less than $\cos\theta$ when $\varepsilon<\varepsilon_0(n,\theta)$.
\end{proof}

\begin{center}
\begin{figure}[h]\label{fig:supersolution}
\includegraphics[width=0.95\textwidth]{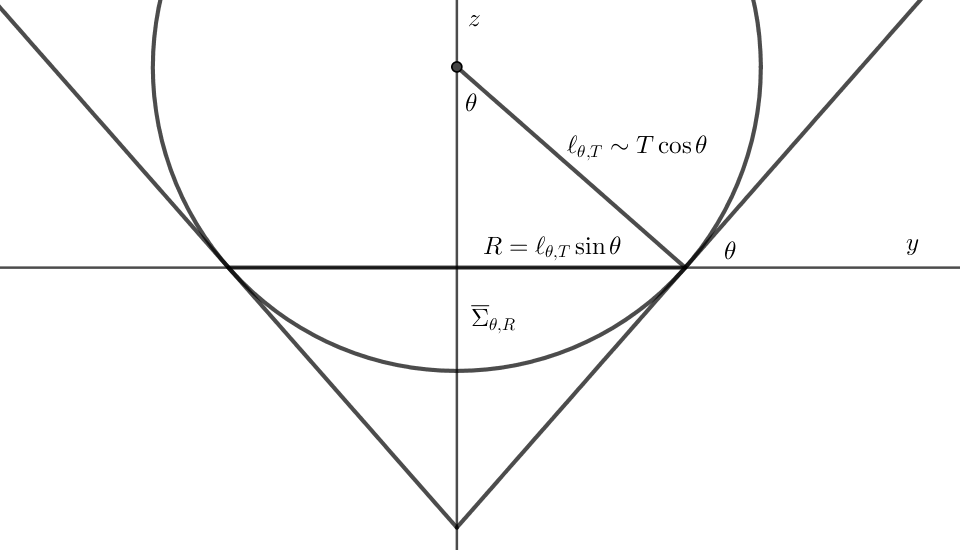}
\caption{Given any $\varepsilon\in(0,\varepsilon_0(n,\theta))$, the portion of $\Pi_{R}$ (the rotated time $T=\sec^2\theta\cosh\left(\frac{R}{\tan\theta}\right)$ slice of the Angenent oval of width $\pi\sec\theta$) lying below height $z=-R\frac{\cos(\theta-\varepsilon)}{\sin\theta}$ is a supersolution of the translator equation when $R>R_\varepsilon:=\frac{2(n-1)}{\varepsilon}$. The surface $\overline \Sigma_{R,\varepsilon}$ is obtained by translating this piece upward so that its boundary lies in $\R^n\times\{0\}$.}
\end{figure}
\end{center}

Consider the `inner' domain
\bann
\overline\Omega_{R,\varepsilon}
:={}&\left\{(x,y)\in S_\theta^{n}:\cos\left(\tfrac{x}{\sec\theta}\right)<\frac{\cosh\left(\tfrac{\sqrt{|y|^2\sin^2\theta+R^2\cos^2(\theta-\varepsilon)}}{\tan\theta}\right)}{\cosh\left(\tfrac{R}{\tan\theta}\right)}\right\}\,.
\eann
Note that $\pd\overline\Omega_{R,\varepsilon}=\pd\overline\Sigma_{R,\varepsilon}$.

The following lemma implies that the inner barrier which touches the outer barrier at $Re_2$ lies above it, so long as $R$ is sufficiently large. 
\begin{lemma}
Given any $R>0$, $\overline\Omega_{\rho_\varepsilon,\varepsilon}\subset\underline \Omega_{R}$, where $\rho_\varepsilon:=\frac{\sin\theta}{\sin(\theta-\varepsilon)}R$.
\end{lemma}
\begin{proof}
It suffices to show that the function $f:\R_+\to\R$ defined by
\[
f(\zeta):=\frac{\cosh\left(\tfrac{\sqrt{\zeta^2\sin^2\theta+\rho_\varepsilon^2\cos^2(\theta-\varepsilon)}}{\tan\theta}\right)}{\cosh\left(\tfrac{\rho_\varepsilon}{\tan\theta}\right)}-\left[\frac{\cosh\left(\tfrac{\zeta}{\tan\theta}\right)}{\cosh\left(\tfrac{R}{\tan\theta}\right)}\right]^{\sin^2\theta}
\]
is non-positive. This follows from log-concavity of the function
\[
g(w):=\cosh\left(\tfrac{\sqrt{w}}{\tan\theta}\right)\,.
\]
Indeed, given any $s\in(0,1)$ and $w>0$, log-concavity of $g$ implies that the function
\[
G(z):=\frac{g(sz+(1-s)w)}{g(z)^s}
\]
is monotone non-decreasing for $z<w$. Since $\zeta<R<\frac{\tan\theta}{\tan(\theta-\varepsilon)}R=\frac{\cos(\theta-\varepsilon)}{\cos\theta}\rho_\varepsilon$, this implies
\bann
\frac{g(\zeta^2\sin^2\theta+\rho_\varepsilon^2\cos^2(\theta-\varepsilon))}{g(\zeta^2)^{\sin^2\theta}}\leq{}&\frac{g(R^2\sin^2\theta+\rho_\varepsilon^2\cos^2(\theta-\varepsilon))}{g(R^2)^{\sin^2\theta}}\\
={}&\frac{g(\rho_\varepsilon^2)}{g(R^2)^{\sin^2\theta}}\,.
\eann
The claim follows.
\end{proof}

\begin{corollary}\label{cor:upperbarrier}
Set $\varepsilon_R:=\frac{2(n-1)}{R}$, $\overline \Sigma_{R}:=\overline \Sigma_{\rho_{\varepsilon_R},\varepsilon_R}$ and $\overline \Omega_{R}:=\overline\Omega_{\rho_{\varepsilon_R},\varepsilon_R}$. Then, for $R>R_0:=\frac{2(n-1)}{\varepsilon_0}$, $\overline\Sigma_{R}$ is a supersolution of the translator equation with boundary $\pd \overline \Sigma_{R}=\pd \overline \Omega_{R}$.
\end{corollary}

\section{Existence} 

In this section we prove the existence theorem, which we now recall.
\begin{theorem*}[Existence of convex translators in slab regions]
For every $n\geq 2$ and every $\theta\in(0,\frac{\pi}{2})$ there exists a strictly convex translator $\Sigma^n_\theta$ which lies in $S_\theta^{n+1}$ and in no smaller slab.
\end{theorem*}
\begin{proof}
Given $R>0$, let $u_{R}$ be the solution of
\begin{equation*}
\begin{cases}\displaystyle
H[u_{R}]=\frac{1}{\sqrt{1+|Du_{R}|^2}}\quad\text{in}\; \Omega_{R}\\
\hspace{17pt} u_{R}= 0\quad\text{on}\; \pd\Omega_{R}\,,
\end{cases}
\end{equation*}
where $\Omega_{R}:=\underline\Omega_{R}$. Since the equation admits upper and lower barriers ($0$ and $\underline u_{R}$, respectively), existence and uniqueness of a smooth solution follows from well-known methods (see, for example, \cite[Chapter 15]{GT}). 
Uniqueness implies that $u_{R}$ is rotationally symmetric with respect to the subspace $\mathbb{E}^{n-1}=\Span\{e_2,\dots,e_{n}\}$. 
Since $\underline u_{R}$ is a subsolution, its graph lies below $\graph u_R$. Since the two surfaces coincide on the boundary $\pd\Omega_{R}$,
\ba\label{eq:Hboundbarrier}
H[u_{R}]=-\inner{\nu_{R}}{e_{n+1}}\geq{}&-\inner{\underline\nu_{R}}{e_{n+1}}\nonumber\\
\geq{}&\cos\theta \cos\left(x\cos\theta\right)\nonumber\\
\geq{}& \cos\theta\left(1-\frac{x}{\frac{\pi}{2}\sec\theta}\right)
\ea
on $\pd\Omega_{R}$, where $\underline\nu_R$ is the downward pointing unit normal to $\graph \underline u_R$. By Corollary \ref{cor:upperbarrier}, we also find, for $R>R_0$, that
\ba\label{eq:heightestimate}
-u_{R}(0)\ge\frac{1-\cos\theta}{\sin\theta}R\to \infty \quad\text{as}\quad R\to\infty\,.
\ea

Let $R_i\to\infty$ be a diverging sequence and consider the translators-with-boundary
\[
\Sigma_i:=\graph u_{R_i}-u_{R_i}(0)e_{n+1}\,.
\]
By Corollary \ref{cor:regularity} and the height estimate \eqref{eq:heightestimate} some subsequence converges locally uniformly in the smooth topology to some limiting translator, $\Sigma$, with bounded second fundamental form. By Theorem \ref{thm:convexity}, $\Sigma$ is convex.

Certainly $\Sigma$ lies in the slab $S_\theta$, so it remains only to prove that it lies in no smaller slab (strict convexity will then follow from the splitting theorem and uniqueness of the Grim Reaper). Set
\[
v:=1-\frac{x}{\frac{\pi}{2}\sec\theta}\,,
\]
where $x(X):=\inner{X}{e_1}$. We claim that
\begin{equation}\label{eq:Hlowerbound}
\inf_{\Sigma\cap\{x>0\}}\frac{H}{v}>0\,.
\end{equation}
Since $\Sigma$ is convex and non-compact (it is complete and translates under mean curvature flow) Myers' theorem implies that $\inf_{\Sigma}H=0$ and hence $\sup_{\Sigma}x= \frac{\pi}{2}\sec\theta$ as desired. To prove \eqref{eq:Hlowerbound}, first observe that
\[
-(\Delta+\cd_V)v=0
\]
and hence
\[
-(\Delta+\cd_V)\frac{H}{v}=|A|^2\frac{H}{v}+2\inner{\cd\frac{H}{v}}{\frac{\cd v}{v}}\,,
\]
where $V$ is the tangential projection of $e_{n+1}$. The maximum principle then yields
\bann
\min_{\Sigma_i\cap\{x>0\}}\frac{H}{v}\geq{}& \min\left\{\min_{\pd\Sigma_i\cap\{x>0\}}\frac{H}{v},\min_{\Sigma_i\cap\{x=0\}}\frac{H}{v}\right\}\\
={}&\min\left\{\cos\theta,\min_{\Sigma_i\cap\{x=0\}}H\right\}\,.
\eann
If $\liminf_{i\to\infty}\min_{\Sigma_i\cap\{x=0\}}H>0$ then we are done. So suppose that $\liminf_{i\to\infty}H(X_i)=0$ along some sequence of points $X_i\in \Sigma_i\cap\{x=0\}$. Then, by Corollary \ref{cor:regularity}, after passing to a subsequence, the translators-with-boundary
\[
\hat \Sigma_i:=\Sigma_i-X_i
\]
converge locally uniformly in $C^\infty$ to a translator (possibly with boundary) $\hat\Sigma$ which lies in $S_\theta$ and satisfies $H\geq 0$ with equality at the origin. Note that the origin must be an interior point since, recalling \eqref{eq:Hboundbarrier}, $H>\cos\theta$ on $\pd\Sigma_i\cap\{x=0\}$ for all $i$. The strong maximum principle then implies that $H\equiv 0$ on $\hat\Sigma$ and we conclude that $\hat\Sigma$ is either a hyperplane or half-hyperplane.
Since, by the reflection symmetry, the limit cannot be parallel to $\{0\}\times\R^{n-1}\times\R$, neither option can be reconciled with the fact that $\hat\Sigma$ lies in $S_\theta$.
\end{proof}

\section{Asymptotics and reflection symmetry}\label{sec:asymptotics}

In this section we prove Theorem \ref{thm:asymptotics}. We begin with proving that, after translation, the translators have the correct asymptotics. 


\begin{theorem}\label{thm:asymptotics2}
Given $n\geq 2$ and $\theta\in(0,\frac{\pi}{2})$ let $\Sigma_\theta^n$ be a convex translator which lies in $S_\theta^{n+1}$ and in no smaller slab. If $n\geq 3$, assume in addition that $\Sigma_\theta^n$ is rotationally symmetric with respect to the subspace $\mathbb{E}^{n-1}:=\Span\{e_2,\dots,e_n\}$. Given any unit vector $\phi\in\mathbb{E}^{n-1}$ the curve $\{\sin\omega \phi-\cos\omega e_{n+1}:\omega\in[0,\theta)\}$ lies in the normal image of $\Sigma_\theta^n$ and the translators
\[
\Sigma^n_{\theta,\omega}:=\Sigma^n_\theta-P(\sin\omega \phi-\cos\omega e_{n+1})
\]
converge locally uniformly in the smooth topology to the oblique Grim hyperplane $\Gamma^n_{\theta,\phi}$ as $\omega\to \theta$, where $P:S^n\to\Sigma^n_{\theta}$ is the inverse of the Gauss map.
\end{theorem}
\begin{proof}
We will first prove the theorem assuming that $\Sigma_\theta^n$ is rotationally symmetric with respect to the subspace $\mathbb{E}^{n-1}:=\Span\{e_2,\dots,e_n\}$. We will then show how this hypothesis can be removed when $n=2$. Fix a unit vector $\phi\in \Span\{e_2,\dots,e_n\}$ and define
\[
\overline \omega:=\sup\{\omega\in[0,\infty):\sin\omega \phi-\cos\omega e_{n+1}\in \nu(\Sigma)\}\,.
\]
Let $\omega_i$ be a sequence of points converging to $\overline \omega$. Then the translators
\[
\Sigma_{i,\phi}:=\Sigma-P_\phi(\omega_i)
\]
have uniformly bounded curvature and pass through the origin. After passing to a subsequence, they must therefore converge locally uniformly to a limit translator. The limit must be the oblique Grim hyperplane $\Gamma^n_{\overline\omega,\phi}$ since it contains the ray $\{r(\cos\overline\omega \phi+\sin\overline\omega e_{n+1}):r>0\}$ and lies in a slab parallel to $S_\theta$ (and, when $n\geq 3$, splits off an additional $n-2$ lines due to the rotational symmetry). In fact, since the components of the normal are monotone along the curve $\gamma(\omega):=P(\sin\omega \phi-\cos\omega e_{n+1})$, the normal must actually converge (to $\sin\overline\omega \phi-\cos\overline\omega e_{n+1}$) along $\gamma$. It follows that the limit is independent of the subsequence and we conclude that the translators
\[
\Sigma_{\omega,\phi}:=\Sigma-P_\phi(\omega)
\]
converge locally uniformly in $C^\infty$ to $\Gamma^n_{\overline\omega,\phi}$ as $\omega\to\overline\omega$. Note that $\overline\omega\leq \theta$ since the limit $\Gamma^n_{\overline\omega,\phi}$ must lie in $S_\theta$. 
It remains to show that $\overline\omega\geq \theta$.

Suppose, to the contrary, that $\overline \omega<\theta$. Given $\omega\in[0,\frac{\pi}{2})$, let $\Pi^\omega_t=\sec\omega\Pi_{\cos^2\omega t}$ be the rotationally symmetric ancient pancake which lies in the slab $S^{n+1}_\omega$ (and no smaller slab) and becomes extinct at the origin at time zero that has been constructed in~ \cite{BLT}. 
The `radius' $\ell_\omega(t)$ of $\Pi^\omega_t$ satisfies
\begin{align}\label{eq:ellomega}
\ell_\omega(t):=\max_{p\in \Pi^\omega_t}\inner{p}{e_2}={}&\sec\omega\,\ell_{0}(\cos^2\omega t)\nonumber\\
={}&-t\cos\omega+(n-1)\sec\omega\log(-t)+c+o(1)
\end{align}
as $t\to-\infty$, where the constant $c$ and the remainder term depend on $\omega$ and $n$. Observe that the ray $L_\omega=\{r(\cos\omega \phi+\sin\omega e_{n+1}):r>0\}$ is tangent to the circle in the plane $\Span\{\phi,e_{n+1}\}$ of radius $-\cos\omega t$ centred at $-te_{n+1}$. Indeed, a point $r(\cos\omega\phi+\sin\omega e_{n+1})$ lies on this circle if and only if
\begin{align*}
\vert r\cos\omega\phi+(r\sin\omega+t)e_{n+1}\vert^2={}&\cos^2\omega t^2\\\Pi^{\omega}_t
\Longleftrightarrow\quad (r-\sin\omega(-t))^2={}&0\,.
\end{align*}
So there exists a unique point with this property, as claimed. Since, by hypothesis, $\theta<\overline \omega$, we conclude from \eqref{eq:ellomega} that the circle of radius $\ell_{\theta}(-t)$ lies above the line $L_{\overline\omega}$ for $-t$ sufficiently large (cf. Figure \ref{fig:pancake_barrier}). Our goal is to show that, in fact, $\Pi^{\omega}_t-te_{n+1}$ lies above $\Sigma$ for $-t$ sufficiently large (and hence $\Pi^{\omega}_t$ lies above $\Sigma_t:=\Sigma+te_{n+1}$ for $-t$ sufficiently large). But since $\Pi^{\omega}_t$ and $\Sigma_t$ both reach the origin at time zero, this would contradict the avoidance principle. 



In order to prove that $\Pi^{\omega}_t-te_{n+1}$ lies above $\Sigma$ for $-t$ sufficiently large, we need an estimate for the `width' of $\Sigma$. We know that, near its `edge region', $\Sigma$ looks like a Grim hyperplane of width $\sec\overline\omega$, whereas, in its `middle region', it looks like two parallel planes of width $\sec\theta$. By convexity, it must lie outside the linearly interpolating region in between (see Figure \ref{fig:linearlyinterpolating}). Lemma~\ref{lem:width} below quantifies this elementary observation.

Given $p\in \Sigma$ 
set
\[
x(p):=\inner{p}{e_1},\quad y(p):=\inner{p}{\phi}\quad \text{and}\quad z(p):=\inner{p}{e_{n+1}}
\]
and, given $h>0$, set
\[
\ell(h):=\max_{p\in \Sigma_h}y(p)\,,
\]
where $\Sigma_h$ is the level set $\Sigma_h:=\{p\in\Sigma:z(p)=h\}$. 

\begin{lemma}[Width estimate]\label{lem:width}
Set
\[
\beta:=\sec\theta-\sec\overline \omega>0\quad\text{and}\quad x_0:=\lim_{\omega\to\overline\omega}x\left(P_\phi(\omega)\right)\,.
\]
For any $\e>0$ there exist $K_\varepsilon<\infty$ and $h_{\e}<\infty$ such that, for every $h>h_{\e}$, $p\in\Sigma_h$ and $s\in[0,1]$,
\bann
0\leq y(p)\le s(\ell(h)-{}&K_\varepsilon)\\
\implies{}& |x(p)-x_0|\geq \tfrac{\pi}{2}\left(\sec\overline\omega+(1-s)(\beta-\tfrac{2}{\pi}x_0)-\e\right)\,.
\eann
\end{lemma}

\begin{center}
\begin{figure}[h]
\includegraphics[width=\textwidth]{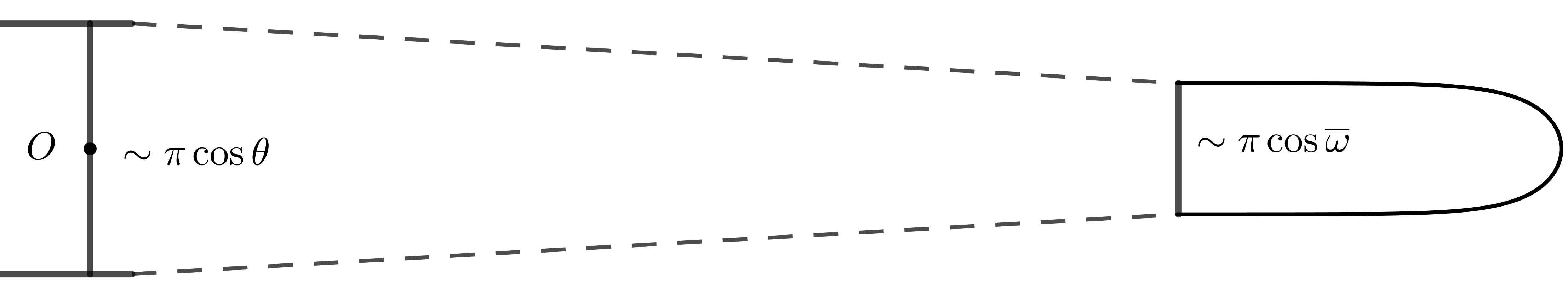}
\caption{Linearly interpolating between the `middle' and `edge' regions in the level set $\Sigma_h$. The horizontal axis is compressed.}\label{fig:linearlyinterpolating}
\end{figure}
\end{center}

\begin{proof}
Choose $\varepsilon>0$. Because $\Sigma$ converges to the oblique Grim hyperplane $\Gamma_{\overline\omega,\phi}$ after translating the `tips', $P_\phi(\omega)$, we can find some $h_\varepsilon$ and $K_\e$ such that
\[
|x(p)-x_0|\ge\tfrac{\pi}{2}\sec\overline\omega-\e
\]
for all $p\in \Sigma_h$ satisfying
\[
0\le y(p)\le \ell(h)-K_\e
\]
so long as $h\geq h_\varepsilon$.
%
Choose some $h\geq h_\varepsilon$ and consider the point $p_1\in \Sigma_h\cap\{e_1,\phi,e_{n+1}\}$ satisfying $x(p_1)\geq x_0$ and $0\leq y(p_1)= \ell(h)-K_\e$. (If there is no such point then the claim is vacuously true, else $p_1$ is uniquely determined.) Then
\[
x(p_1)-x_0\geq \tfrac{\pi}{2}\sec\omega-\e\,.
\]
On the other hand, because $\Sigma$ converges to the boundary of $S_\theta$ after translating vertically, we can assume that $h_\varepsilon$ is so large that
\[
x(p_0)\geq \tfrac{\pi}{2}\sec\theta-\varepsilon
\]
at the point $p_0\in\Sigma_h\cap\mathrm{span}\{e_1,\phi,e_{n+1}\}$ satisfying $y(p_0)=0$ and $x(p_0)\geq x_0$. Since $\Sigma_h$ is convex, we conclude that any point $p\in \Sigma_h\cap\mathrm{span}\{e_1,\phi,e_{n+1}\}$ satisfying $0\le y(p)\le \ell(h)-K_\e$ and $x(p)\geq x_0$ lies beyond the segment joining $p_0$ and $p_1$. In particular, if $y(p)\leq s(\ell(h)-K_\varepsilon)$ then
\begin{align*}
x(p)\geq{}& sx(p_1)+(1-s)x(p_0)\\
\geq{}&s\left(x_0+\tfrac{\pi}{2}\sec\overline\omega-\varepsilon\right)+(1-s)\left(\tfrac{\pi}{2}\sec\theta-\varepsilon\right)\\
={}&x_0+\tfrac{\pi}{2}\left(\sec\overline\omega+(1-s)(\beta-\tfrac{2}{\pi}x_0)\right)-\varepsilon\,.
\end{align*}
The other inequality is proved in much the same way (simply choose the points $p_0$ and $p_1$ on the other side of the $\{x=x_0\}$ plane).
\end{proof}


Reflecting $\Sigma^n$ through the $\{x=0\}$ hyperplane if necessary, we may assume in what follows that $x_0\geq 0$.

Given $\varepsilon>0$, choose $h_\varepsilon$ and $K_\varepsilon$ as in Lemma \ref{lem:width} and consider $h\geq h_{\varepsilon}$. Then, given any $p\in\Sigma$ satisfying $0\leq y(p)<(\ell(h)-K_\varepsilon)$ and $x(p)\geq x_0$, we can choose $s=s(p):=\frac{|y(p)|}{\ell(h)-K_\varepsilon}\in [0,1]$ and hence estimate
\begin{align*}
x(p)\geq{}&\frac{\pi}{2}\left(\sec\theta-\beta\frac{|y(p)|}{\ell(h)-K_\varepsilon}-\e\right)\,.
\end{align*}
Choosing $h_\varepsilon$ larger if necessary, we may assume that $\ell(h_\varepsilon)\geq 2K_\varepsilon$ and hence
\begin{align}\label{eq:widthp}
x(p)\geq{}&\frac{\pi}{2}\left(\sec\theta-\beta\frac{|y(p)|}{\ell(h)}\left(1+\frac{2K_\varepsilon}{\ell(h)}\right)-\e\right)\,.
\end{align}

Note also that, by convexity,
\[
\tan\overline\omega\ge \frac{h}{\ell(h)}\to\tan\overline\omega\;\;\text{as}\;\;h\to\infty\,.
\]

\begin{claim}\label{above}
There exists $\omega\in(\overline\omega,\theta)$ such that for $-t$ sufficiently large, $\Pi^{\omega}_t-te_{n+1}$ lies above $\Sigma$ for $-t$ sufficiently large. 
\end{claim}
\begin{proof}
Arguing by contradiction, suppose that for any $t<0$ and $\omega\in(\overline\omega,\theta)$, there is some point $p\in (\Pi^\omega_{t}-te_{n+1})\cap\Sigma\cap\{X:\inner{X}{\phi}>0\}$. Then there is some $\varphi\in[0,\frac{\pi}{2}]$ such that 
\[
h:=z(p)=-t-\ell_\omega(t)\sin\varphi,\;\; |y(p)|\leq \ell_\omega(t)\cos\varphi\;\; \text{and}\;\; |x(p)|<\frac{\pi}{2}\sec\omega\,,
\]
where $\ell_\omega$ is defined by \eqref{eq:ellomega} (see Figure \ref{fig:pancake_barrier}). Suppose further that $h\geq h_\varepsilon$.
\begin{center}
\begin{figure}[h]
\includegraphics[width=0.8\textwidth]{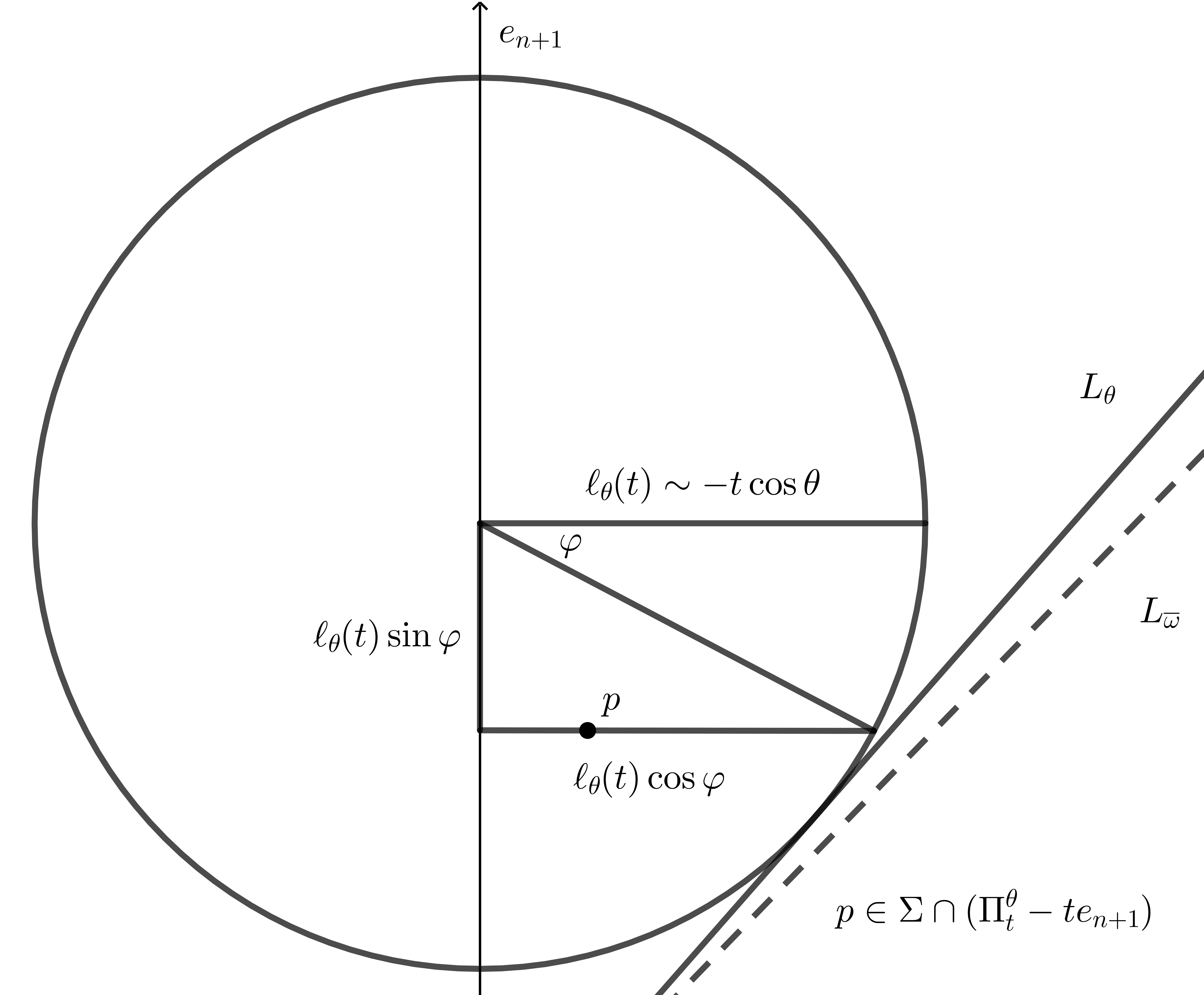}
\caption{If $\overline\omega<\theta$ then the pancake lies above the translator for $-t$ sufficiently large.}\label{fig:pancake_barrier}
\end{figure}
\end{center}
Recalling \eqref{eq:widthp}, we find
\[
\begin{split}
\sec\omega\ge& \sec\theta-\beta\frac{|y(p)|}{h}\frac{h}{\ell(h)}\left(1+\frac{2K_\varepsilon}{h}\frac{h}{\ell(h)}\right)-\e\\
\geq& \sec\theta-\beta\frac{\ell_\omega(t)\cos\varphi}{h}\tan\overline\omega\left(1+\frac{2K_\varepsilon}{h}\tan\overline\omega\right)-\e\,.
\end{split}
\]
That is,
\begin{align*}
\frac{\sec\theta-\sec\omega}{\sec\theta-\sec\overline\omega}\le{}&\frac{\ell_\omega(t)\cos\varphi}{h}\tan\overline\omega\left(1+\frac{2K_\varepsilon}{h}\tan\overline\omega\right)+\frac{\e}{\beta}\\
={}&\frac{\ell_\omega(t)\cos\varphi\tan\overline\omega}{-t-\ell_\omega(t)\sin\varphi}\left(1+\frac{2K_\varepsilon\tan\overline\omega}{-t-\ell_\omega(t)\sin\varphi}\right)+\frac{\e}{\beta}\,.
\end{align*}
Since the right hand side is non-increasing with respect to $\varphi$ for $\varphi\in[0,\frac{\pi}{2}]$, we may estimate
\begin{align*}
\frac{\sec\theta-\sec\omega}{\sec\theta-\sec\overline\omega}\le{}&\frac{\ell_\omega(t)}{-t}\tan\overline\omega\left(1+\frac{2K_\varepsilon}{-t}\tan\overline\omega\right)+\frac{\e}{\beta}\,.
\end{align*}
But $\frac{\ell_\omega(t)}{-t}\to\cos\omega$ as $t\to-\infty$, so we conclude, for $-t\geq -t_\varepsilon$ sufficiently large, that
\begin{align*}
\frac{\sec\theta-\sec\omega}{\sec\theta-\sec\overline\omega}\le{}&\cos\omega\tan\overline\omega+\frac{2\e}{\beta}\le\sin\overline\omega+\frac{2\e}{\beta}\,.
\end{align*}
Choosing $\omega$ sufficiently close to $\overline\omega$ and $\varepsilon$ sufficiently small results in a contradiction. This completes the proof of Claim~\ref{above}. 
\end{proof}

By Claim~\ref{above}, there exists $\omega\in(\overline\omega,\theta)$ such that $\Pi^{\omega}_t-te_{n+1}$ lies above $\Sigma$ for $-t$ sufficiently large. In other words $\Pi^{\omega}_t$ lies above $\Sigma_t:=\Sigma+te_{n+1}$ for $-t$ sufficiently large. But since $\Pi^{\omega}_t$ and $\Sigma_t$ both reach the origin at time zero, this contradicts the avoidance principle. This finishes the proof of Theorem \ref{thm:asymptotics2} in case $n\geq 3$. 

It remains to remove the reflection symmetry hypothesis when $n=2$. If both asymptotic Grim planes have width smaller than $\pi\sec\theta$ then the above argument needs no modification, so it suffices to address the case that $\Sigma$ is asymptotic to the correct oblique Grim plane in one direction, say $-e_2$, but not the other, $e_2$. This can be achieved with a similar argument by centering the ancient pancake not on the $z$-axis but rather on the axis bisecting the two asymptotic lines, i.e. the ray $\{r\left(\cos\frac{\theta-\overline\omega}{2}e_3+\sin\frac{\theta-\overline\omega}{2}e_2\right):r>0\}$. We omit the details since the result in this case was already proved by Spruck and Xiao \cite{SX}.

\end{proof}

Combining the unique asymptotics with the Alexandrov reflection principle, we may complete the proof of Theorem \ref{thm:asymptotics}. 

\begin{corollary}\label{cor:reflex}
Given $\theta\in(0,\frac{\pi}{2})$, let $\Sigma$ be a strictly convex translator which lies in $S_\theta^{n+1}$ and in no smaller slab. If $n\geq 3$ assume in addition that $\Sigma$ is rotationally symmetric with respect to $\mathbb{E}^{n-1}$. Then $\Sigma$ is reflection symmetric across the hyperplane $\{0\}\times \R^{n}$.
\end{corollary}

\begin{proof}
We proceed similarly as in \cite[Theorem 6.2]{BLT}. Let us begin by introducing some notation. Given a unit vector $e\in S^n$ and some $\alpha\in \R$, denote by $\mathrm{H}_{e,\alpha}$ the halfspace $\{p\in \R^{n+1}:\left\langle p,e\right\rangle<\alpha\}$ and by $R_{e,\alpha}\cdot\Sigma:=\{p-2(\left\langle p,e\right\rangle-\alpha)e:p\in \Sigma\}$ the reflection of $\Sigma$ across the hyperplane $\partial \mathrm{H}_{e,\alpha}$. We say that $\Sigma$ \emph{can be reflected strictly about $\mathrm{H}_{e,\alpha}$} if $(R_{e,\alpha}\cdot\Sigma)\cap \mathrm{H}_{e,\alpha}\subset \Omega\cap\mathrm{H}_{e,\alpha}$.

\begin{lemma}[Alexandrov reflection principle]\label{lem:Alexandrov}
Let $\Sigma$ be a convex translator. If $\Sigma_h:=\Sigma\cap\{(x,y,z)\in \R\times \R^{n-1}\times \R:z>h\}$ can be reflected strictly about $\mathrm{H}_{e,\alpha}$ for some $e\in \{e_{n+1}\}^\perp$ then $\Sigma$ can be reflected strictly about $\mathrm{H}_{e,\alpha}$.
\end{lemma}
\begin{proof}
This is a consequence of the strong maximum principle and the boundary point lemma (cf. \cite[Chapter 10]{GT}).
\end{proof}

\begin{claim}\label{claim:reflection}
For every $\alpha\in(0,\frac{\pi}{4})$ there exists $h_\alpha<\infty$ such that $\Sigma_{h_\alpha}:=\Sigma\cap\{(x,y,z)\in \R\times \R^{n-1}\times \R:z>h_\alpha\}$ can be reflected strictly about $\mathrm{H}_{\alpha}:=\mathrm{H}_{e_1,\alpha}$. 
\end{claim}
\begin{proof}
Suppose that the claim does not hold. Then there must be some $\alpha\in (0,\frac{\pi}{4})$ and a sequence of heights $h_i\to\infty$ such that $(R_{\alpha}\cdot\Sigma_{h_i})\cap \mathrm{H}_\alpha\cap \Sigma_{h_i}\neq\emptyset$. Choose a sequence of points $p_i=x_ie_1+y_ie_2\in\Sigma_{h_i}$ whose reflection about the hyperplane $\mathrm{H}_\alpha$ satisfies $(2\alpha-x_i)e_1+y_ie_2\in (R_{\alpha}\cdot\Sigma_{h_i})\cap\Sigma_{h_i}\cap \mathrm{H}_\alpha$ and set $p'_i=x'_ie_1+ y'_ie_2:=(2\alpha-x_i)e_1+y_ie_2$. Without loss of generality, we may assume that $y'_i=y_i\geq 0$. Since $\alpha\leq x_i<\frac{\pi}{2}$, the point $p'_i$ satisfies  $\alpha\geq x'_i>-\frac{\pi}{2}+2\alpha$ so that, after passing to a subsequence, $\lim_{i\to\infty}x'_i\in[-\frac{\pi}{2}+2\alpha,\alpha]$. But since $\Sigma$ is convex and converges, after translating in the plane $\mathrm{span}\{e_2,e_{n+1}\}$, to the Grim hyperplane $\Gamma_{e_2,\theta}$, we conclude that
\[
0=\lim_{i\to\infty}(x_i+x'_i)=2\alpha\,.
\]
So $\alpha=0$, a contradiction.
\end{proof}
It now follows from Lemma \ref{lem:Alexandrov} that $\Sigma$ can be reflected across $\mathrm{H}_\alpha$ for all $\alpha\in(0,\frac{\pi}{2})$. The same argument applies when the halfspace $\mathrm{H}_\alpha$ is replaced by $-\mathrm{H}_{\alpha}=\{(x,y,z)\in \R\times\R\times\R^{n-1}:x>-\a\}$. Taking $\alpha\to 0$ finishes the proof of the corollary.
\end{proof}


\bibliographystyle{acm}
\bibliography{translators}

\end{document}